\documentclass[12pt]{article}

\usepackage{amsfonts}
\usepackage{amsmath}
\usepackage{amsmath}
\usepackage{amssymb}
\usepackage{amsthm}
\usepackage{mathrsfs}
\usepackage{graphicx}
\usepackage[all]{xy}
\usepackage{leftidx}
\usepackage{url}
\usepackage{xcolor}
\usepackage{arydshln}
\usepackage{tikz}

\newtheorem{theorem}{Theorem}[section]
 \newtheorem{proposition}[theorem]{Proposition}
 
 \newtheorem{corollary}[theorem]{Corollary}
 
 \theoremstyle{definition}
 \newtheorem{definition}[theorem]{Definition}
 \newtheorem{example}[theorem]{Example}
 \theoremstyle{remark}
 \newtheorem{remark}[theorem]{Remark}
 \numberwithin{equation}{subsection}

\newcommand{\fp}{{\mathfrak{p}}}

\newcommand{\fg}{{\mathfrak{g}}}

\newcommand{\QQ}{{\mathbb{Q}}}

\newcommand{\FF}{{\mathbb{F}}}
\newcommand{\RR}{{\mathbb{R}}}
\newcommand{\CC}{{\mathbb{C}}}

\newcommand{\GG}{{\mathbb{G}}}
\newcommand{\TT}{{\mathbb{T}}}

\newcommand{\Sp}{{\operatorname{Sp}}}
\newcommand{\symp}{{\operatorname{\mathfrak{sp}}}}

\newcommand{\Gal}[1]{{\operatorname{Gal}(#1)}}

\newcommand{\abs}[1]{{\vert #1 \vert}}

\newcommand{\cl}[1]{{\overline{#1}}}  
\newcommand{\lsup}[2]{{\vphantom{#2}}^{#1}{#2}} 
\newcommand{\oh}[1]{{\lsup{\circ\!}{#1}}}

\newcommand{\ceq}{{\, :=\, }}
\newcommand{\iso}{{\, \cong\, }}

\newcommand{\dia}{{\operatorname{diag}}}
\newcommand{\Nil}{{\operatorname{Nil}}}
\newcommand{\Dst}{{\mathcal{D}^{\operatorname{st}}}}
\newcommand{\Dnil}{{\mathcal{D}_{\operatorname{nil}}}}
\newcommand{\Dent}{{\mathcal{D}_{\operatorname{ent}}}}
\newcommand{\fgnil}{{\fg_{\operatorname{nil}}}}
\newcommand{\fgent}{{\fg_{\operatorname{ent}}}}

\newcommand{\fac}{{\mathcal{F}}}

\title{Fourier Eigenspaces of Waldspurger's Basis}
\author{Aaron Christie}
\date{17 November 2014}
\begin{document}
\maketitle

\abstract{
In this paper we investigate invariant distributions on
$p$-adic $\symp_{2n}$ defined by Waldspurger in \cite{waldtome} and
find the Fourier eigenspaces in their span. We prove that there is a single eigenvalue if
$n$ can be represented as a sum of triangular numbers or is triangular
itself and that none exist otherwise. We determine that the dimension
of this lone eigenspace is equal to the
number of noncommuting representations of $n$ as the sum of at most
two triangular numbers. Each such representation corresponds to what
we call a \emph{Lusztig distribution}. These distributions belong to
the generating set defined by Waldspurger and
form a basis for the eigenspace. Finally, we show that the eigenspace
contains a 1-dimensional subspace consisting of stable distributions
when $n=2\Delta$, $\Delta$ a triangular
number, but otherwise consists of distributions that are not stable.
}


\section{Introduction}

In this paper we study the eigenspaces of the Fourier transform in the
span of certain invariant distributions on symplectic Lie algebras
over a $p$-adic field that appear, denoted by $\phi_\theta(X_T,-)$, in
the book~\cite{waldtome} by Waldspurger as
precursors to the various bases he produces there for the purpose
of studying stability and endoscopy. One of those bases is obtained
from the Fourier transforms of what are called generalized Green functions, which
motivates the question of finding the Fourier eigendistributions in
the span of Waldspurger's distributions. 

Let $\GG$ be a connected reductive group, $\boldsymbol{\fg}$ be its
Lie algebra, both defined over a
$p$-adic field $F$, and let $G = \GG(F)$ and
$\fg=\boldsymbol{\fg}(F)$\footnote{In general, boldface letters denote
  algebro-geometric objects and their unbolded equivalents set-theoretic ones.}. In~\cite{waldtome},
Waldspurger finds a basis for the set $\Dst\cap\Dnil$ of stable invariant distributions supported on
the nilpotent set $\fg_{\operatorname{nil}}$ in $\fg$.  

The set
$\mathcal{D}$ of invariant distributions on $\fg$ consists of elements
of the dual of the convolution algebra $\mathcal{C}^\infty_c(\fg )$ of
locally constant, compactly supported, complex-valued functions on
$\fg$ that are invariant under the action induced by the adjoint
action of $G$ on $\fg$. The set of distributions supported on a
particular subset $\omega$ is denoted $\mathcal{D}(\omega )$, except
in certain cases:
when $\omega$ is $\fgnil$ or $\fgent$ we adopt Waldspurger's notation and
write $\Dnil$ for $\mathcal{D}(\fgnil )$ and $\Dent$ for
$\mathcal{D}(\fgent )$. Harish-Chandra has proved (see \cite{HC}) 
that a basis for $\Dnil$ is given by the nilpotent orbital integrals. Stable distributions are harder to define explicitly,
but $\Dst$ can be briefly described as the closure of the span
of the stable regular semisimple orbital integrals, which have a
straightforward, if lenghty, description\footnote{For which see, \emph{e.g.}, \cite{debackerkazhdan}.}. Aside from the regular
semisimple orbital integrals, determining the stability of a given
distribution (in particular those with nilpotent support) is a
difficult task. This being the case
Waldspurger made use of a bridge from distributions with nilpotent
support and those with support on the ``integral elements'' of $\fg$,
denoted $\fgent$, where more can be said about stability.

The bridge takes the form of a homogeneity result proved earlier by Waldspurger (in~\cite{wald1}),
which states that
\[
\operatorname{res}_{\mathcal{H}}\Dent
=
\operatorname{res}_{\mathcal{H}}\Dnil.
\]
Here, the restriction is to the set 
\[
\mathcal{H} = \displaystyle \sum_{C}\mathcal{C}_c(\fg/\fg_C),
\]
where $\mathcal{C}_c(\fg/\fg_C)$ is the subset of
$\mathcal{C}_c^\infty(\fg)$ consisting of functions invariant under
translation by elements of the parahoric subalgebra $\fg_C$, and the
sum is over all alcoves of the Bruhat-Tits building $\mathcal{B}(\GG
)$ of $\GG$. This allowed him to find his basis in terms of distributions that are
supported on $\fg_{\operatorname{ent}}$, which is composed of the elements of
$\fg$ with integral eigenvalues; alternately,
\[
\fg_{\operatorname{ent}} = \displaystyle \bigcup_{x\in\mathcal{B}(\GG
  )} \fg_{x,0}.
\]

The distributions Waldspurger defines and then restricts to
$\mathcal{H}$ are associated to generalized Green
functions, which are $G$-conjugation invariant functions on
Lie groups defined over finite fields supported on the unipotent elements. Transported via the exponential
map (whose existence, it should be noted, restricts the characteristic
of the finite field to the set of good primes for $G$), they are also defined for Lie algebras and supported on the
nilpotent elements. They can be divided into two
basic types, the majority arising as Deligne-Lusztig characters, the rest
as characteristic functions of cuspidal unipotent character sheaves as
defined by Lusztig (in~\cite{CShI}). The latter kind are called ``Lusztig functions'' by Waldspurger,
and were originally distinguished by their special properties
\emph{vis \`{a}
vis} the Fourier transform: they are, up to scalar multiplicaton, the
only nilpotently supported functions on the Lie algebra whose Fourier
transform is also supported on $\fg_{\operatorname{nil}}$. We show
that these functions are tied to the Fourier eigenfunctions in
Waldspurger's proto-basis, though the correspondence is not one-to-one.

In this paper, we deal solely with the case where $\GG = \operatorname{Sp}_{2n}$ and
$\boldsymbol{\fg} = \symp_{2n}$ over a $p$-adic field $F$. We briefly
recall the definition of Lusztig functions along with some facts about
the parametrization of nilpotent orbits needed to support the definition.
We go on to consider the inflation of generalized Green functions to
$\fg$ and show that among these functions the Fourier
eigenfunctions are inflations of products of Lusztig functions, which
leads us to define and enumerate the class of $p$-adic Lusztig
functions. These functions do not correspond bijectively with the set
of Lusztig functions---there are always more $p$-adic Lusztig
functions than there are Lusztig functions. In fact, $p$-adic Lusztig
functions exist on $p$-adic $\symp_{2n}$ over a $p$-adic field $F$
when no such functions exist on $\symp_{2n}$
over the residue field of $F$.

Finally, we recall the definition of Waldspurger's distributions,
$\phi_\theta(X_T,-)$, and
investigate the Fourier eigenspaces in their span. Our main result is
the following:
\begin{quote}
{\bf Theorem~\ref{thm:eigendistributions}:} Let $\fg = \symp_{2n}(F)$. Then
\begin{enumerate}
\item $\widehat{\phi}_\theta(X_T,-) \in \operatorname{span}\{\phi_\theta(X_T,-)\}_{(\theta ,X_T)}$ if and only if $T$
is trivial (\emph{i.e.}, $X_T=0$).
\item If $n$ be a triangular number or the sum of two triangular
numbers, $\operatorname{span}\{\phi_\theta(X_T,-)\}_{(\theta ,X_T)}$
contains a single Fourier eigenspace $E$ with eigenvalue 
\[
\left(\operatorname{sgn}(-1) \displaystyle  q^{-\frac{1}{2}}\sum_{x\in\FF_q}
\operatorname{sgn}(x)\psi (x)\right)^{n}.
\]
Moreover, up to
$G$-conjugacy,
\[
\operatorname{dim}(E) = \begin{cases} 2(d_1(8n+2) - d_3(8n+2)) & \textrm{$n$ not
triangular,} \\ 2(d_1(8n+2)-d_3(8n+2)+1) & \textrm{$n$
triangular},\end{cases} 
\]
where $d_i(m)$ denotes the number of divisors of $m$ congruent to $i$ mod 4.
\end{enumerate} 
\end{quote}
We additionally show that the eigenspace $E$ contains a
1-dimensional subspace consisting of stable distributions precisely
when $n$ is twice a triangular number. To conclude, we describe a
conjectural geometrization of the eigendistributions in Theorem
\ref{thm:eigendistributions}. 
\vskip.5cm

{\bf Acknowledgements.} The author is pleased to thank Dr. Clifton
Cunningham and Dr. Monica Nevins,
with both of whom the author has had many useful conversations who
have been extremely patient and helpful readers of early drafts.

\section{Preliminaries}\label{sec:prelims}

\subsection{Notation}

Thoughout, $F$ is a non-archimedean local field with ring of integers
$\mathcal{O}_F$, uniformizer $\varpi$, and finite residue field
$\FF_q$ of characteristic $p>0$. We will take the valuation on
$F$ to be normalized so that $v_F(\varpi )=1$. Certain results, including the
parametrization of nilpotent orbits, the existence of a Killing form
and of a $G$-equivariant isomorphism between the unipotent subvariety of
the Lie group and the nilpotent subvariety of the Lie algebra require that $p$ not be too
small. In particular, the latter two require that $p$ be a good prime
(for which see \cite{springerunipotent} and \cite{springersteinberg}).

Even though some results (particularly background ones) hold more
generally, we will only consider the linear algebraic group
$\mathbb{G} = \operatorname{Sp}_{2n}$ over $F$ and its Lie algebra $\boldsymbol{\fg} =
\operatorname{\mathfrak{sp}}_{2n}$, unless declared
otherwise. Let $G = \GG (F)$ and
$\fg = \boldsymbol{\fg}(F)$. If we wish to consider $\GG$, $G$, 
$\boldsymbol{\fg}$, or $\fg$ over alternate fields, like an
algebraic closure $\cl{F}$ of $F$ or the residue field $\FF_q$, that field will be included as a
subscript: $\GG(\cl{F}) = G_{\cl{F}}$, $\boldsymbol{\fg}(\cl{F}) = \fg_\cl{F}$, and so on. 

Although for the most part it is unnecessary, there will be some instances
where it will be useful to choose a representation for $G$ or
$\fg$. When we do, $(V, q_V)$ will be a $2n$-dimensional vector space
$V$ over $F$ with symplectic form $q_V$ given by $q_V( x,y) = x^tJy$, where $J = \begin{bmatrix}0
  &I_n\\-I_n&0\end{bmatrix}$. Therefore, after embedding in
$\operatorname{GL}_{2n}(F)$, $G(V) = G = \{g\in\operatorname{GL}_{2n}(F)\mid
g^tJg=J\}$ and $\fg(V) = \fg =
\{h\in\operatorname{\mathfrak{gl}}_{2n}(F)\mid h^tJ - Jh = 0\}$.

\subsection{Nilpotent Orbits}\label{sec:nilpotentorbits}

This section contains a brief recollection of the parametrization of
nilpotent orbits in semisimple Lie algebras, which plays a role in the
definition of Lusztig functions. All of the material here
is modelled on the paper by Nevins,~\cite{nevins}.

First, recall that as long as $F$ has characteristic 0 or $p>3(h-1)$,
$h$ the Coxeter number of $G$ ($h = 2n$ in the case considered here), as a consequence of Jacobson-Morozov
theory (see, \emph{e.g.},~\cite{collmcg}), the nilpotent
$\GG$-orbits of a finite semisimple Lie algebra $\boldsymbol{\fg}$ over an
algebraically closed field (which henceforward will be called
\emph{geometric orbits}) can be parametrized by appropriate partitions
of $\dim V=2n$. A \emph{partition} of a number $c$ is a non-increasing
sequence $\lambda = (\lambda_1, \lambda_2, \dots ,\lambda_d)$ of
positive integers such that $\lambda_1+ \cdots +\lambda_d =c$. For any $1\leq
i\leq n$, let $m_i = \abs{\{\lambda_j \in \lambda \mid \lambda_j =
  i\}}$ be the \emph{multiplicity} of $i$. In the case of $\fg = \symp_{2n}(\cl{F})$,
the admissible partitions are those where $m_i$ is even for every odd
$i$. In the correspondence, the parts of a partition $\lambda$ give the sizes of the
blocks in the Jordan normal form of the elements in the corresponding
nilpotent orbit variety, which will be denoted $\mathcal{O}_\lambda$.

Since every such orbit variety contains an $F$-rational point we can
consider the set of $F$-rational $G_F$-orbits within
$\mathcal{O}_\lambda (F)$. For $\fg =
\operatorname{\mathfrak{sp}}_{2n}(F)$, these orbits
are parametrized by isometry classes of quadratic forms, as in the
following proposition.

\begin{proposition}[\cite{nevins}, Prop. 5.1]
Let $\lambda$ be a partition of $2n$ and let $m_j$ be the multiplicity
of $j$ in $\lambda$. Assume that $m_j$ is even whenever $j$ is
odd. The $G_F$-orbits in $\mathcal{O}_\lambda(F)$ are parametrized by
$n$-tuples
\[
\cl{Q} = (Q_2,Q_4,\dots ,Q_{2n}),
\]
where $Q_j$ represents the isometry class of a nondegenerate quadratic
form over $F$ of dimension $m_j$ ($Q_j = 0$ if $m_j =0$).
\end{proposition}

By the Witt Decomposition Theorem, every
nondegenerate quadratic form can be decomposed as a direct sum
\[
Q = q_0^m \oplus Q_{aniso},
\]
where $m\leq\frac{1}{2}\dim{Q}$, $q_0^m$ is the $m$-fold direct sum of
the hyperbolic plane, which has matrix representative
\[
q_0 = \begin{bmatrix}0&1\\1&0\end{bmatrix},
\]
and $Q_{aniso}$ is $Q$'s anisotropic part (recall that a quadratic
form is \emph{isotropic} if there is an $x\neq 0$ such that $Q(x) = 0$ and
\emph{anisotropic} otherwise). Below, we reproduce a
table (adapted from~\cite{nevins}), which lists diagonal matrix representatives for all
isometry classes of quadratic forms by dimension. The quadratic forms
$Q_{2i}$ in the proposition above can be chosen from among those in
the table (replacing $\varpi^{-1}$ with $1$ when working in $\fg_{\FF_q}$).

\begin{figure}[h]
\begin{center}
\begin{tabular}{|l|l|}
 \hline 
 Dimension & Representatives of Isometry Classes  \\
\hline
1 & $[t]$, \qquad $t\in\{1,\epsilon ,\varpi^{-1},\epsilon\varpi^{-1}\}$ \\
2&$\dia(t,\epsilon t )$, \qquad $t\in\{1,\varpi^{-1}\}$ \\
&$\dia(t\varpi^{-1},-t^\prime )$, \qquad $t,t^\prime\in\{1,\epsilon\}$ \\
3&$\dia(1 , -\epsilon , t\varpi^{-1} )$, \qquad $t\in\{1,\epsilon\}$\\
&$\dia(\varpi^{-1},-\epsilon\varpi^{-1}  , t )$, \qquad $t\in\{1,\epsilon\}$\\
4&$\dia(1,-\epsilon , -\varpi^{-1} , \epsilon\varpi^{-1} )$  \\
\hline
\end{tabular}
\caption{Representatives of isometry classes of quadratic forms listed
  by dimension. The
  element $\epsilon$ is a nonsquare unit in $\mathcal{O}_F$.}
\end{center}
\end{figure}

Henceforward, we will denote by $\Nil(\fg )$ the set of
rational nilpotent orbits in $\fg$, and identify those elements with pairs $(\lambda ,
\cl{Q})$ consisting of a partition of $2n$ and an $n$-tuple of quadratic
form representatives. Geometric orbits will simply be
identified with partitions of $2n$. 

\begin{example}\label{ex:reps}
For reasons that will be made clear in the next section, we will be interested in rational orbits contained in the
geometric orbit matching a partition of the form $\lambda = (2i, 2i-2,
\dots , 4,2)$ for some $i$ or the union of two such partitions,
meaning $\lambda = (2i,2i-2,\dots ,2j+2,2j,2j,2j-2,2j-2,\dots
,2,2)$ for $i\geq j$. Representatives in the first case follow a
simple pattern that the following low-dimensional examples should make
clear.
\\
\[
\begin{array}{rc}
\symp_{2}(F)\colon &
\begin{bmatrix}0&Q_2\\0&0\end{bmatrix},
\\
\symp_{6}(F): & \left[\begin{array}{c:cc|c:cc}0&&&Q_2 &&\\ \hdashline
     &0&1&&0&0\\ &0&0&&0& -Q_4 \\ \hline 0&&&0&&\\ \hdashline
     &0&0&&0&0\\ &0&0&&-1&0 \end{array}\right], \\
\end{array}
\]
where $Q_2,Q_4,Q_6\in\{1,\epsilon,\varpi^{-1},\epsilon\varpi^{-1}\}$.
In these cases, where each part in $\lambda$ has multiplicity 1,
only 1-dimensional (in particular, only
anisotropic) quadratic forms appear.

For partitions of the second type we can look at $\symp_4(F)$, which contains a geometric orbit parametrized by
$\lambda = (2,2)$. Rational orbit representatives there take the form
\[
\left[\begin{array}{c:c|c:c}0&&Q_2&\\ \hdashline &0&&Q_2^\prime\\ \hline
  0&&0&\\ \hdashline &0&&0\end{array}\right]\textrm{ \ or \ } 
\begin{bmatrix}0&0&0&1\\0&0&1&0\\0&0&0&0\\0&0&0&0\end{bmatrix} 
\]
where $Q_2,Q_2^\prime\in\{1,\epsilon,\varpi^{-1},\epsilon\varpi^{-1}\}$.
In $\symp_8(F)$, there is a geometric orbit parametrized by $\lambda =
(4,2,2)$. There, most rational orbit representatives take the form
\[
\left[\begin{array}{c:ccc|c:ccc}
0&&&&Q_2&&&\\ \hdashline &0&0&0&&Q_2^\prime&0&0\\  &0&0&1&&0&0&0\\
&0&0&0&&0&0&Q_4\\ \hline 0&&&&0&&& \\ \hdashline &0&0&0&&0&0&0 \\
&0&0&0&&0&0&0 \\ &0&0&0&&0&-1&0\end{array}\right] 
\]
where, again, all entries are drawn from $\{1,\epsilon,\varpi^{-1},\epsilon\varpi^{-1}\}$.

\end{example}

\subsection{Lusztig Functions}\label{subsec:lusztigfunctions}

Throughout this section all Lie groups and algebras are
assumed to be defined over a finite field $\FF_q$. 

In his Utrecht paper, \cite{luszfourier}, Lusztig investigated the existence of
nilpotently supported, $\cl{\QQ}_\ell$-valued functions on semisimple finite Lie algebras
defined over $\FF_q$ whose Fourier
transform is also nilpotently supported. Employing his theory of character
sheaves~\cite[\emph{et seq.}]{CShI}, he determined that such functions are actually quite
rare (\cite[Cor.\! 10]{luszfourier}). Up to scalar multiplication there is at most one such function
on $\symp_{2n}(\FF_q)$ and it exists exactly when $n$ is a triangular
number (\emph{i.e.}, $n= i(i+1)/2$ for some $i$).

So, let us fix one of these scalar multiples and denote it by $\lsup{\circ\!}{f}$, in imitation
of Waldspurger in~\cite{waldtome}, where these functions are officially called
``Lusztig functions.'' If more than one Lusztig function appears, we
will distinguish them with subscripts: $\oh{f}_{2n}$, and so on. These
functions are supported on the geometric orbit
parametrized by the partition of $\lsup{\circ}{\lambda} \ceq
(2i,2i-2,\dots ,4,2)$ of $2n$. These will be called the \emph{Lusztig
  orbit}  and \emph{Lusztig partition} of $\symp_{2n}$,
respectively. Finally, since they will be relevant after shifting
attention to the $p$-adic field $F$, we call any partition that is a
union of two Lusztig partitions an \emph{extended Lusztig partition}.


As in section \ref{sec:nilpotentorbits}, rational orbits contained in
a Lusztig orbit are represented by tuples of quadratic form
representatives, either $[1]$ or $[\epsilon ]$, where $\epsilon$ is a nonsquare
in $\FF_q^\times$.  Let 
$e[(\lambda,\cl{Q})]$ be the characteristic function of the orbit
corresponding to the pair $(\lambda ,\cl{Q})\in\Nil(\fg_{\FF_q})$, and define
\begin{equation*}\label{eq:formdeterminant}
\lsup{\circ}{\operatorname{sgn}}(\cl{Q}) = \displaystyle
\prod_{\stackrel{i\textrm{ even}}{i/2 \textrm{ odd}}}
\operatorname{sgn}\circ (-1)^{\lfloor
  \operatorname{dim}(Q_i)/2\rfloor}\det(Q_i),
\end{equation*}
where $\det$ means the class of the determinant of $Q_i$ in
$\FF_q^\times/(\FF_q^\times)^2$. Then,
\begin{equation*}\label{eq:ohfdefinition}
\lsup{\circ\!}{f}_{2n} = \displaystyle \sum_{(\lsup{\circ}{\lambda},\cl{Q})\in\operatorname{Nil}(\fg (\FF_q))}
\lsup{\circ}{\operatorname{sgn}}(\cl{Q})e[(\lsup{\circ}{\lambda},\cl{Q})].
\end{equation*}

Lusztig shows that---even more than having Fourier
transform supported on $\fgnil$---these functions are actually eigenfunctions of
the Fourier transform. Their eigenvalues have been determined by
Waldspurger \cite[Prop. V.8]{waldtome}: For any $x\in\FF_q$, let
\[
\operatorname{sgn}(x)  = \begin{cases}0 & \textrm{if }x=0,\\ 1&
  \textrm{if }
  x\in\left(\FF_q^\times\right)^2,\\-1 & \textrm{if } x\in\FF_q^\times -
  \left(\FF_q^\times\right)^2. \end{cases}
\]
Then, the eigenvalue of the Lusztig function on $\symp_{2n}(\FF_q)$ under
the Fourier transform (defined with respect to a character $\psi\colon\FF_q\rightarrow\CC^\times$) is
\[
\oh{\gamma}_{2n} = \left(\operatorname{sgn}(-1) \displaystyle  q^{-\frac{1}{2}}\sum_{x\in\FF_q}
\operatorname{sgn}(x)\psi (x)\right)^{\frac{k(k+1)}{2}}.
\]

\vskip .5cm

\begin{example}\label{ex:lusztigfns}
\begin{enumerate}
\item In the case of $\symp_{2}(\FF_q)$, the Lusztig function is 
\[
\oh{f} = e\left[\begin{bmatrix}0&1\\0&0\end{bmatrix}\right] -
e\left[\begin{bmatrix}
   0 &\epsilon\\0&0\end{bmatrix}\right],
\]
the two terms being the characteristic functions of the rational orbit of the
given representative.
\item For $\symp_{6}(\FF_q)$, the Lusztig function is
\[
\oh{f} = e[(1,1)] + e[(1,\epsilon )] - e[(\epsilon ,1)] - e[(\epsilon
, \epsilon )].
\]
\end{enumerate}
\end{example}

\section{$p$-adic Lusztig Functions}\label{sec:plusztigfunctions}

Our approach to finding the eigendistributions among Waldspurger's
distributions involves relating the known
properties of generalized Green functions on finite Lie algebras to their inflations to 
$\fg$ over $F$, $F$ our $p$-adic field. The obvious way of
mediating between these settings is through the parahoric subalgebras of
$\fg$ and their associated reductive quotients. 

For any point
$x\in\mathcal{B}(\GG )$ of the Bruhat-Tits building of $\GG$, we
denote by $G_x$ and $\fg_x$ respectively the parahoric subgroup of $G$ and
parahoric subalgebra of $\fg$ associated to $x$ (see
\cite{bruhattits84} or \cite{rabinoff}). Both depend only on
the facet $\fac$ of $\mathcal{B}(\GG )$ containing $x$, so we will also substitute
$\fac$ for $x$ in the notation, as in $G_\fac$ and $\fg_\fac$. If $r\in\RR$, we denote the
corresponding Moy-Prasad filtration subgroups and subalgebras (see \cite{moyprasad}) by
$G_{x,r}$ and $\fg_{x,r}$. Finally, we write $\cl{G}_x =
G_{x,0}/G_{x,0+}$ and $\cl{\fg}_{x} = \fg_{x,0}/\fg_{x,0+}$ for the
reductive quotients of $G$ and $\fg$ attached to $x$ (or a facet
containing $x$).


\subsection{Inflations}\label{subsec:inflations}

Generalized Green functions are defined in \cite[Ch. II]{waldtome},
where they are denoted $Q_T$. Aside from mentioning that they come in
two broad types, Deligne-Lusztig characters on one hand and Lusztig
functions on the other, we will not reprise their definition here since finding the Fourier
eigendistributions among the distributions Waldspurger defines only
requires knowing the eigenfunctions among them. It is known, as a
byproduct of the proof of
\cite[Prop. II.8]{waldtome}, that these are precisely the
Lusztig functions.

\begin{definition}\label{def:compositeplusztig}
Let $\fac$ be a facet of $\mathcal{B}(\GG)$ such that 
\[
\cl{\fg}_\fac \iso \symp_{2\Delta_1}(\FF_q)\times\symp_{2\Delta_2}(\FF_q
)\times\cdots\times\symp_{2\Delta_m}(\FF_q),
\]
where $\Delta_i$ is triangular for all $i$. Let
\[
\lsup{\circ\!}{f}_\fac(Y) = \begin{cases}
  \left(\lsup{\circ\!}{f}_{2\Delta_1}\times\lsup{\circ\!}{f}_{2\Delta_2}\times\cdots\times\lsup{\circ\!}{f}_{2\Delta_m}\right)\circ\rho_{\fac,0}(Y), & Y\in\fg_{\fac,0}\\ 0, &
  Y\not\in\fg_{\fac,0}, \end{cases}
\]
where $\rho_{\fac ,0}\colon\fg_{\fac ,0}\rightarrow\cl{\fg}_{\fac}$ is
the reduction map. We call all such functions \emph{$p$-adic Lusztig functions}.
\end{definition}

The following proposition further justifies the name. In preparation,
we fix a Fourier transform on $\fg$. Fix a Killing form\footnote{The
  existence of such requires that $p\neq 2$.} $\langle \
\ , \ \rangle\colon \fg\times\fg \rightarrow F$ for $\fg$ and a
nontrivial additive character $\psi\colon F\rightarrow \CC^\times$
that is trivial on the maximal ideal $\fp_F$ of $F$. Then, for any
$f\in\mathcal{C}^\infty_c(\fg )$,
\[
(\mathscr{F}f)(X) = \widehat{f}(X) = \displaystyle \int_\fg \psi
(\langle X,Y\rangle )f(X)\mathrm{d}X,
\]
with respect to the Haar measure normalized so that $\operatorname{meas}({\fg}_{x,0+}) = \abs{\cl{\fg}_{x}}^{-\frac{1}{2}} =
\abs{\FF_q}^{-\frac{\operatorname{dim}(\cl{\fg}_x)}{2}}$ for all
$x\in\mathcal{B}(\GG )$.

\begin{proposition}\label{prop:compositeplusztig}
In the setting of Definition~\ref{def:compositeplusztig},
$\lsup{\circ\!}{f}_\fac$ is an eigenfunction of the Fourier transform on
$\fg$.
\end{proposition}
\begin{proof}
We begin by observing that in fact no more than two factors are
required in Definition \ref{def:compositeplusztig} to capture all
$p$-adic Lusztig functions. As noted by J.-K. Yu
in~\cite[$\S7.6$]{Yusmoothmodels}, if $\Phi$ is the set of roots of $G$ with
respect to a maximal $F$-torus $T$, the root data of the reductive quotient
$\cl{\fg}_\fac$ is $(X^*(T),\Phi_\fac ,X_*(T),\Phi^\vee_\fac )$, where
the subset $\Phi_\fac$ of $\Phi$ associated to the facet $\fac$ is the
set of roots $\alpha\in\Phi$ such that there is an affine
root $\alpha + n$ whose reflecting hyperplane contains $\fac$ (or, more
precisely, the smallest affine space containing $\fac$ in the
apartment determined by $T$). As Yu says, ``this allows us to see the isomorphism class of
[$\cl{\fg}_\fac$] by pure thought when
$G$ is split adjoint or split simply connected and $\fac$ is
a vertex.'' We find ourselves in exactly this case, since we are indeed
dealing with a split simply connected group, and if the reductive
quotient is to have the shape required by Definition~\ref{def:compositeplusztig}, its root system
$(\mathcal{A},\Phi_\fac,\mathcal{A}^\vee,\Phi_\fac^\vee )$ must be semisimple (\emph{i.e.},
$\abs{\Phi_F} = \dim{\mathcal{A}}$). This cannot be the case if $\fac$ has positive
dimension, since then the affine space containing $\fac$ must be spanned
by some nonempty subset of $\Phi$. Thus, $\fac$ must be a vertex.

From this, it follows that every reductive quotient of $\fg$ that
carries a Lusztig function has the form $\symp_{2\Delta_1}(\FF_q )\times
\symp_{2\Delta_2}(\FF_q)$, where $n=\Delta_1 + \Delta_2$ and $\Delta_1$
and $\Delta_2$ are triangular numbers (taking $\symp_{0}(\FF_q)$ to be
the trivial algebra). To see this, consider the extended Dynkin
diagram of type $C_n$:
\[
 \xy
\POS (-10,0)  ="z",
\POS (0,0) *\cir<3pt>{} ="a",
\POS (10,0) *\cir<3pt>{} ="b",
\POS (20,0) *\cir<3pt>{} ="c",
\POS (30,0) *\cir<3pt>{} ="d",
\POS (40,0) *\cir<3pt>{} ="e",
\POS (50,0) *\cir<3pt>{} ="f",
\POS (60,0) ="g",
\POS "z" \ar@{}^<<<{} "a",
\POS "a" \ar@{=>}^<<<{\alpha_{0}} "b",
\POS "b" \ar@{-}^<<<{\alpha_{1}} "c",
\POS "c" \ar@{.}^<<<{\alpha_{2}} "d",
\POS "d"\ar@{.}^<<<<{\alpha_{l}} "e",
\POS "e" \ar@{<=}^<<<{\alpha_{n-1}} "f",
\POS "f" \ar@{}^<<<<{\alpha_{n}} "g"
\endxy
\]
The roots corresponding to the nodes are $\alpha_0 = -2e_1$, $\alpha_n
= 2e_n$, and $\alpha_i = e_{i} - e_{i+1}$, $1\leq i \leq n-1$. The
Dynkin diagrams resulting from deleting a node from the extended
diagram tell us the isomorphism classes of reductive quotients at vertices.
Removing the $0^\textrm{th}$ or $n^\textrm{th}$ node gives the Dynkin diagram
of type $C_n$, corresponding to a reductive quotient isomorphic to
$\symp_{2n}(\FF_q)$. Removing the $i^\textrm{th}$ node for $1\leq i
\leq n-1$ gives a disconnected diagram corresponding to
$\symp_{2i}(\FF_q)\times\symp_{2(n-i)}(\FF_q)$.

Therefore, we can assume $\oh{f}_\fac$ is the inflation of a Lusztig
function or the product of two Lusztig functions. With the Haar
measure on $\fg$ normalized as above, we have that
$\mathscr{F}(\oh{f}_\fac ) =
\left(\mathscr{F}(\oh{f}_{2\Delta_1}\times\oh{f}_{2\Delta_2})\right)_{\fac}
=
\mathscr{F}(\oh{f}_{2\Delta_1}\times\oh{f}_{2\Delta_2})\circ\rho_{\fac
  ,0}$, where $\rho_{\fac , 0}$ is the restriction map; \emph{i.e.}, the Fourier
transform commutes exactly with inflation. This follows from \cite[Prop. 1.13]{cunninghales}. It suffices
to show, then, that $\oh{f}_{\Delta_1}\times\oh{f}_{\Delta_2}$ is an
eigenfunction of the Fourier transform on
$\symp_{2\Delta_1}(\FF_q)\times\symp_{2\Delta_2}(\FF_q)$.

By the convolution theorem, $\mathscr{F}(\oh{f}_{2\Delta_1}\times\oh{f}_{2\Delta_2}) =
c(\mathscr{F}(\oh{f}_{2\Delta_1}\times 1))\ast (\mathscr{F}(1\times\oh{f}_{2\Delta_2}))$
for some constant $c$. With normalizations as above, $c=1$, so we have
\[
\begin{aligned}
\mathscr{F}(\oh{f}_{2\Delta_1}\times\oh{f}_{2\Delta_2})  &= (\mathscr{F}(\oh{f}_{2\Delta_1}\times 1))\ast
(\mathscr{F}(1\times\oh{f}_{2\Delta_2}))(s,t) \\ & =
\displaystyle \int \mathscr{F}(\oh{f}_{2\Delta_1}\times 1)(\sigma ,\tau
)\mathscr{F}(1\times\oh{f}_{2\Delta_2})(s-\sigma ,t-\tau )d\sigma d\tau \\
&= \displaystyle \int \delta_0(\tau )\mathscr{F}(\oh{f}_{2\Delta_1})(\sigma
)\delta_0(s-\sigma )\mathscr{F}\oh{f}_{2\Delta_2})(t-\tau )d\sigma d\tau \\
&= \mathscr{F}(\oh{f}_{2\Delta_1})(s)\cdot\mathscr{F}(\oh{f}_{2\Delta_2})(t)\\
&= \oh{\gamma}_{2\Delta_1}\oh{\gamma}_{2\Delta_2}\oh{f}_{2\Delta_1}\oh{f}_{2\Delta_2},
\end{aligned}
\]
where $\oh{\gamma}_{2\Delta_1}$ and $\oh{\gamma}_{2\Delta_2}$ are the eigenvalues of
$\oh{f}_{2\Delta_1}$ and $\oh{f}_{2\Delta_2}$ respectively, which exist by
\cite[Cor. 10]{luszfourier}. With that, we are done.
\end{proof}

\begin{remark}\label{rem:compositeeigenvalues}
Note that, from the formula for the eigenvalues proved by Waldspurger
(recalled in $\S\ref{subsec:lusztigfunctions}$), $\oh{\gamma}_{2\Delta_1}$
and $\oh{\gamma}_{2\Delta_2}$ differ only in their exponents, which means that the
eigenvalue of $\oh{f}_{2\Delta_1}\times\oh{f}_{2\Delta_2}$ is
\[
\oh{\gamma}_{2\Delta_1}\oh{\gamma}_{2\Delta_2} =  \left(\operatorname{sgn}(-1) \displaystyle  q^{-\frac{1}{2}}\sum_{x\in\FF_q}
\operatorname{sgn}(x)\psi (x)\right)^{\Delta_1 +\Delta_2}.
\]
With our normalizations, this is the eigenvalue of the inflation of
$\oh{f}_{2\Delta_1}\times\oh{f}_{2\Delta_2}$ to $\fg$ as well.
\end{remark}

\begin{proposition}\label{prop:plusztigenum}
The distinct nontrivial $p$-adic Lusztig functions on $\fg =
\symp_{2n}(F)$ can be classified up to conjugacy by the following conditions on $n$:
\begin{enumerate}
\item If $n$ is triangular, there exist two
  $p$-adic Lusztig functions on $\fg$ up to conjugacy, corresponding to
 $G$-conjugacy classes of hyperspecial vertices in $\mathcal{B}(\Sp_{2n} )$.
\item For each representation of $n$ as a sum of \emph{distinct}
  triangular numbers, $\Delta_1$ and $\Delta_2$, there exist two $p$-adic Lusztig
  functions on $\fg$.
\item If $n=2\Delta$, $\Delta$ a triangular number, there exists one
  $p$-adic Lusztig function on $\fg$ obtained by inflation from the
  reductive quotient $\symp_n\times\symp_n$.
\end{enumerate}
Therefore, the dimension of the $p$-adic Lusztig functions up to
conjugacy is equal to the number of noncommuting representations of
$n$ as a sum of less than three triangular numbers. 
\end{proposition}
\begin{proof}
\begin{enumerate}
\item $n$ is triangular: Given the foregoing, it is clear that the
  functions here arise as inflations of the Lusztig functions that
  exist on $\symp_{2n}(\FF_q)$, and thus correspond to vertices
  that give this reductive quotient, \emph{i.e.}, the hyperspecial
  vertices. 
 The only point that might require
  justification is the implicit assertion that there are only two
  conjugacy classes of hyperspecial vertices in
  $\mathcal{B}(\Sp_{2n})$. This is proved in~\cite[see
  pg. 3]{shemanske}, and correspond to the first and last nodes of the
  extended Dynkin diagram in the proof of Proposition~\ref{prop:compositeplusztig}.
\item $n = \Delta_1 + \Delta_2$: In light
  of~\ref{prop:compositeplusztig}, the claim here is just that
  whenever $n$ has this form there are two vertices such that 
\[
\cl{\fg}_{x} = \symp_{2\Delta_1}(\FF_q)\times\symp_{2\Delta_2}(\FF_q).
\]
From \cite[$\S7.6$]{Yusmoothmodels}, and as can be seen
in the proof of Proposition~\ref{prop:compositeplusztig}, these two vertices
exist: they correspond to the $\Delta_1+1^\textrm{th}$ and $\Delta_2+1^\textrm{th}$ nodes in
the extended Dynkin diagram.
\item $n=2\Delta$: Same argument as above, yielding in this case only
  one vertex.
\end{enumerate}
\end{proof}

\begin{example}
\begin{enumerate}
\item The Lusztig function on $\symp_{2}{}(\FF_q)$ is 
\[
\oh{f}_2 = e\left[\begin{bmatrix}0&1\\0&0\end{bmatrix}\right] -
e\left[\begin{bmatrix}
    0&\epsilon\\0&0\end{bmatrix}\right].
\]

In this case there are two conjugacy classes of parahorics available
to inflate this function through. Working in the standard apartment,
we take $x_0$ and $x_1$ to be hyperspecial vertices corresponding to
these two classes of parahorics, 
\[
\begin{bmatrix}0&1\\0&0\end{bmatrix},
\quad \begin{bmatrix}0&\epsilon\\0&0\end{bmatrix},
\quad \begin{bmatrix}0&\varpi^{-1}\\0&0\end{bmatrix},
\quad \begin{bmatrix}0&\epsilon\varpi^{-1}\\0&0\end{bmatrix}
\]
as representatives for the four distinguished rational nilpotent orbits in the
Lusztig orbit of $\symp_{2}(\cl{F})$, and write $e[(1)]$, $e[(\epsilon )]$,
  $e[(\varpi )]$, and $e[(\epsilon\varpi )]$ for the
  characteristic functions of those orbits. Then the $p$-adic Lusztig functions on
  $\symp_{2}$ over $F$ coming from this apartment are
\[
\begin{aligned}
\oh{f}_{x_0} =& \begin{cases} e[(1)] - e[(\epsilon )] & \textrm{on
  }\symp_2(F)_{x_0}\\ 0 & \textrm{on }\symp_2(F)\setminus\symp_2(F)_{x_0} \end{cases}\\
\oh{f}_{x_1} =& \begin{cases} e[(\varpi^{-1})] - e[(\epsilon\varpi^{-1} )] & \textrm{on
  }\symp_2(F)_{x_1}\\ 0 & \textrm{on }\symp_2(F)\setminus\symp_2(F)_{x_1} \end{cases}.
\end{aligned}
\]
\item The Lusztig function on $\symp_{6}(\FF_q)$ is
\[
\oh{f} = e[(1,1)] + e[(1,\epsilon )] - e[(\epsilon ,1)] - e[(\epsilon
, \epsilon )].
\]
Again there are two relevant conjugacy classes of parahorics
associated to hyperspecial vertices,
denoted $\symp_{6,x_0}$ and $\symp_{6,x_1}$.
The distinguished orbits for each parahoric are:
\begin{itemize}
\item $\symp_{6,x_0}{}\colon (1,1),(1,\epsilon ),(\epsilon ,1),(\epsilon ,
  \epsilon )$
\item $\symp_{6,x_1}{}\colon (\varpi^{-1} ,\varpi^{-1} ),(\varpi^{-1} ,\epsilon\varpi^{-1}
  ),(\epsilon\varpi^{-1} ,\varpi^{-1} ),(\epsilon\varpi^{-1} ,
  \epsilon\varpi^{-1} )$
\end{itemize}
so the two $p$-adic Lusztig functions are
\begin{equation*}
\oh{f}_{x_0} = e[(1,1)]+e[(1,\epsilon )]-e[(\epsilon ,1)]-e[(\epsilon ,
  \epsilon )] 
\end{equation*}
\begin{multline*}
\oh{f}_{x_1} =  e[(\varpi^{-1} ,\varpi^{-1} )]+e[(\varpi^{-1} ,\epsilon\varpi^{-1} )]\\-e[(\epsilon\varpi^{-1} ,\varpi^{-1} )]-e[(\epsilon\varpi^{-1} ,
  \epsilon\varpi^{-1} )] 
\end{multline*}
on their respective parahorics, and trivial outside them.
\end{enumerate}
\end{example}

\section{The Main Result}\label{sec:mainresult}

In this section, $\GG$ and $\boldsymbol{\fg}$ are defined over the
$p$-adic field $F$.

Recall the following definitions from the introduction: $\mathcal{C}^\infty_c(\fg)$ is the
$\CC$-algebra of locally constant,
compactly supported complex-valued functions on $\fg$, $\mathcal{D}$
is the space of invariant distributions on $\fg$ (those elements of
$(\mathcal{C}^\infty_c(\fg))^\vee$ that are invariant under the action of
$G$), $\mathcal{D}(\omega )$ is the subset of elements of
$\mathcal{D}$ supported on the subset $\omega$.  Finally, the subset
$\mathcal{H}$ is
\[
\mathcal{H} = \displaystyle \sum_{C}\mathcal{C}_c(\fg/\fg_C),
\]
the sum over all alcoves $C$ of $\mathcal{B}(\GG )$. Each summand is
the subset of $\mathcal{C}^\infty_c(\fg )$ consisting of elements
invariant under translation by elements of the parahoric subalgebra $\fg_C$.

\subsection{Distributions From Inflations of Generalized Green
  Functions}\label{subsec:distributions}

We begin by describing the distributions Waldspurger associates to
generalized Green functions. In passing from $\FF_q$ to the
$p$-adic field $F$, he makes use of a parameter set $\Theta (V)$
consisting of 5-tuples of combinatorial data. The
eigenfunctions are those indexed by $\theta =
(k^\prime,k^{\prime\prime},\varnothing,\varnothing,\varnothing )$ where
$k^\prime (k^\prime + 1) + k^{\prime\prime}(k^{\prime\prime}+1) =
2n$, together with an element $X_T$ of the torus associated to the
corresponding generalized Green function, $Q_T$. 

To describe the distributions without wading into the intricacies
of $\Theta (V)$, we will say that $\theta\in\Theta (V)$ consists of
the following data:
\begin{itemize}
\item An orthogonal decomposition $V = V_0\oplus V_1$;
\item a maximal
  unramified torus $\TT\leq \GG (V_1)$ defined over $F$;
\item a facet $\fac$ of $\mathcal{B}(\GG (V_0))$ such that there
  exists a $p$-adic Lusztig function $\oh{f}_\fac$, \emph{i.e.}, such
  that there is a Lusztig function or product of Lusztig functions defined on $\cl{\fg
    (V_0)}_\fac$. 
\end{itemize}

Given such a parameter, choose an element $X_T\in \mathfrak{t} =
\operatorname{Lie}(\TT )(F)$
``integral of regular reduction,'' \emph{i.e.}, with
$Z_G(X_T) =T$ and whose image in $\cl{\fg}_\fac$ has
centralizer equal to the image of $\mathfrak{t}$ in $\cl{\fg}_\fac$.

Then, with Haar measure normalized as described in
Section~\ref{subsec:inflations} define, for every
$f\in\mathcal{C}^\infty_c(\fg )$,
\[
\phi_\theta (X_T,f) = \displaystyle \abs{\cl{G}_F}^{-\frac{1}{2}}\abs{\cl{\fg}_F}^{-\frac{1}{2}}\int_{T\backslash G}\int_{\fg
  (V_0)}f(x^{-1}(X_T + Y)x)\oh{f}_\fac (Y)\mathrm{d}Y\mathrm{d}x,
\]
where $\oh{f}_\fac$ is the $p$-adic Lusztig function that comes
packaged with $\theta$. Let $\Phi = \{\phi_\theta (X_T,-)\mid
\theta\in\Theta (V), X_T\in\mathfrak{t}^\textrm{\,i.r.r.}\}$. By
definition, $\Phi\subseteq\mathcal{D}_{\operatorname{ent}}$. Waldspurger goes on to
prove that they no longer
depend on the choice of $X_T$ by the time they are restricted to
$\mathcal{H}$ (see~\cite[Cor. III.5]{waldtome}). Moreover, those
restrictions, denoted $\phi_\theta^\mathcal{H}$, form a
basis for $\operatorname{res}_{\mathcal{H}}\mathcal{D}(\fg_{\operatorname{ent}})$
(\cite[Cor. III.10(ii)]{waldtome}). 
\subsection{Fourier Eigendistributions in
  $\operatorname{span}\Phi$}

The Fourier transform of a distribution $D\in\mathcal{C}^\infty_c(\fg)^\vee$ is defined to be 
\[
\widehat{D}(f) = D(\widehat{f}).
\]

\begin{theorem}\label{thm:eigendistributions}
Let $\fg = \symp_{2n}(F)$. Then
\begin{enumerate}
\item $\widehat{\phi}_\theta(X_T,-) \in \operatorname{span}\Phi$ if and only if $T$
is trivial (\emph{i.e.}, $X_T=0$).
\item If $n$ be a triangular number or the sum of two triangular
numbers, $\operatorname{span}\Phi$
contains a single Fourier eigenspace $E$ with eigenvalue 
\[
\left(\operatorname{sgn}(-1) \displaystyle  q^{-\frac{1}{2}}\sum_{x\in\FF_q}
\operatorname{sgn}(x)\psi (x)\right)^{n}.
\]
Moreover, up to
$G$-conjugacy,
\[
\operatorname{dim}(E) = \begin{cases} 2(d_1(8n+2) - d_3(8n+2)) & \textrm{$n$ not
triangular,} \\ 2(d_1(8n+2)-d_3(8n+2)+1) & \textrm{$n$
triangular},\end{cases} 
\]
where $d_i(m)$ denotes the number of divisors of $m$ congruent to $i$ mod 4.
\end{enumerate}
\end{theorem}
\begin{proof}
\begin{enumerate} \item Let $f\in\mathcal{C}^\infty_c(\fg )$. Then
$\widehat{\phi}_\theta(X_T,f) = \phi_\theta (X_T,\widehat{f})$. By
\cite[Prop. I.9]{waldtome}, $\widehat{f}$ is supported on the set of
topologically nilpotent elements of $\fg$. Since $X_T$ is semisimple with regular
reduction, the region of integration for $\phi_\theta$ does not
intersect this set unless $X_T=0$; hence
$\operatorname{res}(\widehat{\phi}_\theta (X_T, -))=0$ when $X_T\neq
0$. Since the
$\phi^\mathcal{H}_\theta$ form a basis for
$\operatorname{res}_{\mathcal{H}}(\mathcal{D}_{\operatorname{ent}})$,
if $\widehat{\phi}_\theta (X_T,-)$ were to lie in
$\operatorname{span}\Phi$, it would be trivial, which is not the case.

On the other hand, suppose $X_T=0$, which can only be true if $T$
itself is trivial. Then, for any $f\in\mathcal{C}^\infty_c(\fg )$,
\[
\begin{aligned}
\widehat{\phi}_\theta (0,f) & = \displaystyle \abs{\cl{G}_F}^{-\frac{1}{2}}\abs{\cl{\fg}_F}^{-\frac{1}{2}}\int_{G}\int_{\fg
  }\widehat{f}(x^{-1}Yx)\oh{f}_\fac (Y)\mathrm{d}Y\mathrm{d}x
\\
& = \displaystyle \abs{\cl{G}_F}^{-\frac{1}{2}}\abs{\cl{\fg}_F}^{-\frac{1}{2}}\int_{G}\int_{\fg
  }f(x^{-1} Yx)\widehat{\oh{f}}_\fac (Y)\mathrm{d}Y\mathrm{d}x
\\
& = \displaystyle \abs{\cl{G}_F}^{-\frac{1}{2}}\abs{\cl{\fg}_F}^{-\frac{1}{2}}\int_{G}\int_{\fg
  }f(x^{-1}Yx)\oh{\gamma}\oh{f}_\fac (Y)\mathrm{d}Y\mathrm{d}x
\\
& = \displaystyle \oh{\gamma}\abs{\cl{G}_F}^{-\frac{1}{2}}\abs{\cl{\fg}_F}^{-\frac{1}{2}}\int_{G}\int_{\fg
  }f(x^{-1}Yx)\oh{f}_\fac (Y)\mathrm{d}Y\mathrm{d}x \\
& = \oh{\gamma}\phi_\theta (0,f).
\end{aligned}
\]
The third equality follows from
Proposition~\ref{prop:compositeplusztig}.
\item The first statement follows from the proof above together with
  the observations in Remark \ref{rem:compositeeigenvalues}. The rest
  follows from Proposition~\ref{prop:plusztigenum}
together with the fact that the number of representations of
$n$ as a sum of two triangular numbers is equal to the number of
representations of $8n+2$ as a sum of two odd squares, which is proved
in~\cite{grosswald}, as is the fact that this number is equal to
$d_1(8n+2)-d_3(8n+2)$ when it is not 0. 
\end{enumerate}
\end{proof}

In light of the relation the eigendistributions have to Lusztig
functions and the fact that they bear similar properties, we make the
following definition.

\begin{definition}\label{def:lusztigdist}
The distributions in
$\operatorname{span}\Phi$ with $X_T=0$
are called \emph{Lusztig distributions}.
\end{definition}

\begin{corollary}\label{cor:eigenbasis}
A basis for the Fourier eigenspace $E$ in
$\operatorname{span}\Phi$ is given
by the Lusztig distributions.
\end{corollary}

Hence, the original generating set defined by Waldspurger contains a
basis for its Fourier eigenspace.

\section{Stable Eigendistributions}\label{sec:stable}

Since much of \cite{waldtome} is concerned with stability of distributions, it's worth
asking whether any of the Fourier eigendistributions described in
Theorem \ref{thm:eigendistributions} are stable. Indeed there are stable
eigendistributions, though they are relatively rare, and correspond to
the $p$-adic Lusztig functions described in case 3 of Proposition
\ref{prop:plusztigenum}.

\begin{proposition}
The Fourier eigenspace $E$ of $\operatorname{span}\Phi$ on $\symp_{2n}(F)$ contains a 1-dimensional subspace
consisting of stable distributions exactly when $n = 2\Delta$,
$\Delta$ a triangular number.
\end{proposition}
\begin{proof}
Establishing this without referring to Waldspurger's parameter sets is
impossible, but we will avoid explaining them in full here. The reader
is referred to \cite[III.1, IV.1]{waldtome} to verify our
claims about them.

According to \cite[Thm. IV.13]{waldtome}, a basis for
$\operatorname{res}_{\mathcal{H}}(\Dst\cap \mathcal{D}_{\operatorname{ent}})$ is
given by a set of distributions denoted $\phi^{\mathcal{H}}_{\xi
  ,\overline{1}}$, where the parameter $\xi$ comes from
$\Xi^{\operatorname{st}}(V)$. This parameter corresponds to possibly
more than one 5-tuple $(k^\prime
,k^{\prime\prime},\mu^0,\mu^\prime,\mu^{\prime\prime}) =
\theta\in\Theta (V)$ such that $k^\prime = k^{\prime\prime}$, and $\phi^{\mathcal{H}}_{\xi
  ,\overline{1}}$ is a linear combination of $\phi^\mathcal{H}_\theta$
for those $\theta$. If any of these $\phi^\mathcal{H}_\theta$ is
restricted from a Fourier eigendistribution in
$\operatorname{span}\Phi$, then $\theta = (k^\prime
,k^\prime ,\varnothing ,\varnothing , \varnothing )$ where $k^\prime$
is triangular and $k^\prime (k^\prime +1) = n$. For any
$\xi\in\Xi^{\operatorname{st}}(V)$, there is at most one such $\theta$,
which, characterized as in $\S\ref{subsec:distributions}$,
corresponds to
\begin{itemize}
\item the trivial decomposition $V = V$,
\item $\mathbb{T} = \{1\}$,
\item a facet $\fac$ of $\mathcal{B} (\GG )$ such that $\cl{\fg}_\fac
  \iso
  \symp_{2k^\prime}(F)\times\symp_{2k^\prime}(F)$.
\end{itemize}
For $\phi_\theta (X_T,-)$ with such a $\theta$ we must have
$X_T=0$. Therefore, $\phi^\mathcal{H}_{\xi,\overline{1}}$ is the
restriction of exactly one distribution from
$\operatorname{span}\Phi$: $\phi_\theta (0,-)$
associated to a $p$-adic Lusztig function of the kind appearing in
case 3 of Proposition \ref{prop:plusztigenum}, which exists if and
only if $n=2\Delta$ where $\Delta$ is triangular. In this case,
$\Delta = k^\prime$.
\end{proof} 

\section{Geometrization}\label{sec:geo}

We conclude with some comments concerning a question that naturally
emerges considering the origins of Lusztig functions noted in $\S
\ref{subsec:lusztigfunctions}$---namely, that they are of geometric
origin, arising as characteristic functions of cuspidal unipotent
character sheaves. Is the same is true of $p$-adic
Lusztig functions? Do there exist character sheaves defined on some
variety over $F$ from which $p$-adic Lusztig functions can be
recovered? Or, since we're speaking of character sheaves, it might be
better to ask the same questions of Lusztig distributions.

In either case, the answer appears to be yes. Here is a sketch of how such a picture
might be realized\footnote{Seeing as this last section is informal and
  conjectural, we elect to skip any definitions not already given, all
  of which can be found
  in, \emph{e.g.}, the author's thesis \cite{myphd}.}: Let $\mathcal{O}_{\oh{\lambda}}$ be a nilpotent
orbit variety over $F$ matching an extended Lusztig partition
$\oh{\lambda}$ (an open subvariety of $\symp_{2n}$). Every rank 1 $\cl{\QQ}_\ell$-local system
$\mathcal{L}$ ($\ell$ a prime different from $p$) on the
smooth part\footnote{(which is the orbit minus a point, so we will
  simply refer to it as if it were the entire orbit here)} of
$\mathcal{O}_{\oh{\lambda}}$ is equivalent to an $\ell$-adic character
\[
\chi_{\mathcal{L}}\colon\operatorname{\pi_1}(\mathcal{O}_{\oh{\lambda}},\cl{x})\longrightarrow
\cl{\QQ}_\ell^\times
\]
of the \'etale fundamental group of $\mathcal{O}_{\oh{\lambda}}$. If
$\oh{\lambda} = \oh{\lambda}_1\cup\oh{\lambda}_2$ is a decomposition
of $\oh{\lambda}$ as a union of two Lusztig partitions, this
fundamental group is isomorphic to the product
\[
\left(\operatorname{\pi_1}\left(\mathcal{O}_{\oh{\lambda}_1,\cl{\FF}_q},\cl{x}\right)\times
  \operatorname{\pi_1}\left(\mathcal{O}_{\oh{\lambda}_2,\cl{\FF}_q},\cl{x}\right)\right)\times\Gal{\cl{F}/F},
\]
where $\mathcal{O}_{\oh{\lambda}_i,\cl{\FF}_q}$ is the nilpotent orbit
variety for the Lusztig partition $\oh{\lambda}_i$ over the algebraic closure of
the residue field $\FF_q$ of $F$ and $\cl{F}$ is a separable closure of $F$.

The character $\chi_{\oh{f}_i}$ of
$\operatorname{\pi_1}\left(\mathcal{O}_{\oh{\lambda}_i,\cl{\FF}_q},\cl{x}\right)$
equivalent to the local system whose characteristic function is the
Lusztig function $\oh{f}_i$ supported on
$\mathcal{O}_{\oh{\lambda}_i,\cl{\FF}_q}(\cl{\FF}_q)$ is known (and given by Waldspurger in
\cite[II.IV, II.V]{waldtome}). There are at most two conjugacy classes
of $p$-adic Lusztig functions obtained by inflating $\oh{f}_1\times\oh{f}_2$. These
two functions should be recoverable from the local systems on
$\mathcal{O}_{\oh{\lambda}}$ equivalent to the characters
$\left(\chi_{\oh{f}_1}\otimes\chi_{\oh{f}_2}\right)\otimes\operatorname{id}$ and
$\left(\chi_{\oh{f}_1}\otimes\chi_{\oh{f}_2}\right)\otimes\operatorname{sgn}(\varpi )$ of the fundamental
group of $\mathcal{O}_{\oh{\lambda}}$, where $\operatorname{id}$ is
the trivial character of $\Gal{\cl{F}/F}$ and
$\operatorname{sgn}(\varpi )$ is the character of $\Gal{\cl{F}/F}$
that factors through the nontrivial character of the quotient group
$\Gal{F(\sqrt{\varpi})/F}$.

Recovering a $p$-adic Lusztig function $\oh{f}_{\fac}$ from such a
local system requires the existence of an appropriate integral model,
$\underline{\mathcal{O}}_{\oh{\lambda},\fac}$, which will be matched
to one or the other of the local systems above, whose
$\mathcal{O}_F$-points are $\mathcal{O}_{\oh{\lambda}}(F)\cap\fg_{\fac
  ,0}$. When such a model exists the characteristic function of the
nearby cycles of the local system will be $\oh{f}_1\times\oh{f}_2$, from which
$\oh{f}_{\fac}$ can be obtained by inflation. A key feature of this
machinery is the Galois characters distinguishing the two local
systems, which give an action of $\Gal{\cl{F}/F}$ on the nearby
cycles, a sheaf on a variety \emph{over $\FF_q$}. This action makes it
impossible to obtain a characteristic function from the nearby cycles
when they are taken with respect to an integral model that is not
matched to the local system on $\mathcal{O}_{\oh{\lambda}}$.

A detailed recounting of this process, as well as how it can be used
to produce distributions, can be found in the
author's thesis \cite{myphd}, where its validity is proved for
$\symp_{2}$. The only obstruction to establishing this picture in
general for the symplectic case (and the odd orthogonal case, which is
quite similar) is proving the existence of the
appropriate integral models and carrying out the nearby cycles
calculations for them. It seems this could be accomplished by proving
the existence of a weak Neron model (see \cite{BLR}, $\S3.5$ Def.1)
for $\mathcal{O}_{\oh{\lambda}}$. The author presumes this
can be done but is ignorant of such a proof at present. Beyond idle
curiosity, an answer to this question would be a meaningful step in
extending recent work of Cunningham and Roe in \cite{cunningroe} that expands Lusztig's
geometric approach to character theory beyond the realm of finite fields.

\bibliography{sources}
\bibliographystyle{amsplain}

\end{document}